\documentclass{article}
\usepackage[a4paper]{geometry}
\usepackage[all,cmtip]{xy}
\usepackage[utf8]{inputenc}
\usepackage[initials]{amsrefs}
\usepackage[labelsep=period]{caption}
\usepackage{amsfonts,authblk,bm,doi,enumitem,fancyhdr,float,graphicx,amssymb,amsthm,mathrsfs,mleftright,relsize,sectsty,subcaption,tikz,titlesec}
\usepackage[fleqn]{amsmath}
    
\newtheorem{theorem}{Theorem}[section]
\newtheorem{lemma}[theorem]{Lemma}

\theoremstyle{definition}

\theoremstyle{plain}

\newtheorem{innerthm}{Theorem}
\newenvironment{maintheorem}[1]
  {\renewcommand\theinnerthm{#1}\innerthm}
  {\endinnerthm}

\numberwithin{equation}{section}
\numberwithin{figure}{section}

\renewcommand{\geq}{\geqslant}
\renewcommand{\leq}{\leqslant}

\setlist[enumerate]{leftmargin=20pt,itemsep=0pt,topsep=0pt}
\setlist[enumerate,1]{label=\emph{(\roman*)}}
\setlist[itemize]{leftmargin=20pt,itemsep=0pt,topsep=0pt}

\makeatletter
\renewcommand\section{\@startsection {section}{1}{\z@}%
                                   {-3.5ex \@plus -1ex \@minus -.2ex}%
                                   {1.3ex \@plus.2ex}%
                                   {\normalfont\large\scshape}}
\makeatother

\title{ \vspace{-6ex}\bf \large Enumerating tame friezes over $\mathbb{Z}/n\mathbb{Z}$\footnotetext{
		
		
		Supported by EPSRC grant EP/W002817/1 (IS, MvS) and EPSRC DTP grants EP/W524098/1 (AZ) and EP/T518165/1 (SB, internship).
		}}

\makeatletter
\renewcommand*{\@fnsymbol}[1]{\hspace*{-10pt}}
\makeatother

\author{Sammy Benzaira, Ian Short, Matty van Son, and Andrei Zabolotskii}

\date{\vspace{-5ex}}

\usepackage[markup=underlined,todonotes={textsize=footnotesize}]{changes}
\usepackage{color,soul}

\begin{document}

\maketitle

\begin{abstract}
We use a class of Farey graphs introduced by the final three authors to enumerate the tame friezes over $\mathbb{Z}/n\mathbb{Z}$. Using the same strategy we enumerate the tame regular friezes over $\mathbb{Z}/n\mathbb{Z}$, thereby reproving a recent result of B\"ohmler, Cuntz, and Mabilat.
\end{abstract}

\section{Introduction}

Our objective here is to enumerate the tame friezes over the ring $\mathbb{Z}/n\mathbb{Z}=\{0,1,\dots,n-1\}$. To achieve this, we use the correspondence between tame friezes and paths in a class of graphs termed \emph{Farey graphs} by the final three authors in \cite{ShSoZa2024}. This approach gives us relatively short and simple arguments with a geometric flavour. 

A \emph{frieze} over $\mathbb{Z}/n\mathbb{Z}$ is an array of finitely many bi-infinite rows of elements of $\mathbb{Z}/n\mathbb{Z}$ offset alternately as in Figure~\ref{fig1}, with $0$'s on the first and last rows, and with the property that any diamond of four contiguous entries satisfies the rule $ad-bc=1$ in $\mathbb{Z}/n\mathbb{Z}$. 

\begin{figure}[ht] \centering
	\begin{subfigure}[b]{0.7\textwidth}
 \centering
	\(
		\vcenter{
			\xymatrix @-0.9pc @!0 {
				& 0 && 0 && 0 && 0 && 0 && 0 && 0 && 0 && 0 && 0 & \\
				&& 1 && 1 && 1 && 1 && 1 && 1 && 1 && 1 && 1 && \\      
\raisebox{-11pt}{\ldots}\,  & 1 && 2 && 4 && 3 && 1 && 2 && 4 && 3 && 1 && 2 &   \,\raisebox{-11pt}{\ldots}  \\
				  && 1 && 2 && 1 && 2 && 1 && 2 && 1 && 2 && 1 && \\
                   & 1 && 3 && 4 && 2 && 1 && 3 && 4 && 2 && 1 && 3   \\
				&& 2 && 3 && 2 && 3 && 2 && 3 && 2 && 3 && 2 && \\
				& 0 && 0 && 0 && 0 && 0 && 0 && 0 && 0 && 0 && 0 &    \\
			}
		}
	\)
	\end{subfigure}\quad
		\begin{subfigure}[b]{0.2\textwidth}\centering	
\(\vcenter{\
	\xymatrix @-1.0pc @!0 {
		&b &\\
		a && d\\
		& c &
	}
}\)
	\end{subfigure}
	\caption{A tame frieze over $\mathbb{Z}/5\mathbb{Z}$ of width 6 (left) and a diamond of four entries (right)}
        \label{fig1}
\end{figure}
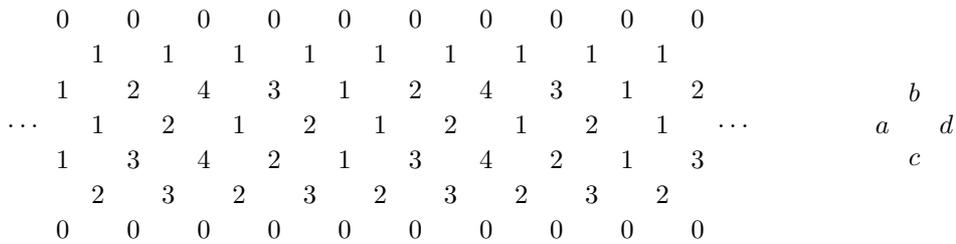 

The \emph{width} of the frieze is the number of rows minus one. The frieze is \emph{regular} if the second and second-last rows comprise 1's only. These definitions of `width' and `regular' are consistent with \cite{ShSoZa2024} but at odds with some other literature. A frieze is \emph{tame} if any diamond of nine contiguous entries has determinant 0. See \cite{ShSoZa2024} for  formal definitions of  these concepts.

There has been significant interest in enumerating friezes over finite rings recently. In \cite{Mo2021}, Morier-Genoud enumerated the regular tame friezes over any finite field. This result was reproved in \cite{ShSoZa2024} where tame friezes over finite fields (not necessarily regular) were also enumerated. A string of works by B\"ohmler, Cuntz, and Mabilat have enumerated the regular tame friezes over $\mathbb{Z}/n\mathbb{Z}$; the most recent and comprehensive of these works are \cite{BoCu2024,Ma2024}. These authors also consider other problems related to enumerating friezes and they consider other finite rings. Here we present two results: the first on enumerating tame friezes over $\mathbb{Z}/n\mathbb{Z}$ and for the second we offer a concise proof of the recently discovered enumeration of tame regular friezes. 

We denote by $\nu_p(n)$ the $p$-adic valuation of $n$, which is the highest power of the prime $p$ in the prime factorisation of $n$. The products in both theorems are taken over prime divisors of $n$. 

\begin{maintheorem}{A}\label{theoremA}
The number of tame friezes of width $m$ over the ring $\mathbb{Z}/n\mathbb{Z}$ is
\[
\prod_{p|n}\frac{p^{(\nu_p(n)-1)(m-1)}(p^{m-1}+(-1)^m)(p-1)}{p+1}.
\]
\end{maintheorem}

For the second theorem, following Morier-Genoud \cite{Mo2021}, we write
\[
[k]_{q} = \frac{q^{k}-1}{q-1}\quad\text{and}\quad \binom{k}{2}_{\!\!q} = \frac{(q^k-1)(q^{k-1}-1)}{(q-1)(q^2-1)},
\]
where $q$ is an integer greater than 1 (and $[k]_q=k$ if $q=1$).

\begin{maintheorem}{B}\label{theoremB}
The number of regular tame friezes of width $m$ over the ring $\mathbb{Z}/n\mathbb{Z}$ is
\[
\prod_{p|n}\Phi_m\big(p^{\nu_p(n)}\big),
\]
where $\Phi_m(p^r) =p^{(r-1)(m-3)}[k]_{p^2}$ for $m=2k+1$, and for $m=2k$, 
\[
\Phi_m(p^r) =
\begin{cases}
p^{(r-1)(m-3)}(p-1)\displaystyle\binom{k}{2}_{\!\!p} & \text{for $k$ even, $p\neq 2$,}\\
p^{(r-1)(m-3)}\mleft(\!(p-1)\displaystyle\binom{k}{2}_{\!\!p}+p^{k-1}-1\!\mright) & \text{for $k$ even, $p=2$, $r\neq 1$,}\\
p^{(r-1)(m-3)}\mleft(\!(p-1)\mleft(\!\displaystyle\binom{k}{2}_{\!\!p}+[r-1]_{p^{2-k}}\!\mright)+p^{k-1}\!\mright) & \text{otherwise.}
\end{cases}
\]
\end{maintheorem}

\section{Farey graphs}

We will use just one class of Farey graphs from \cite{ShSoZa2024}, namely the directed graphs $\mathscr{E}_n$ associated to the rings $\mathbb{Z}/n\mathbb{Z}$. These graphs are double covers of the 1-skeletons of Platonic graphs ($\mathscr{E}_3$ covers the tetrahedron, $\mathscr{E}_4$ the octahedron, and so forth). The vertices of $\mathscr{E}_n$ are pairs $(a,b)$, where $a,b\in\{0,1,\dots,n-1\}$ and $\gcd(a,b,n)=1$, and there is a a directed edge from vertex $(a,b)$ to  vertex $(c,d)$ if $ad-bc=1$ in $\mathbb{Z}/n\mathbb{Z}$. We denote the vertex $(a,b)$ by a formal fraction $a/b$. 
On some occasions we represent a vertex $a/b$ by $a'/b'$ for some pair $a'$ and $b'$ of integers congruent to $a$ and $b$; for example, we often write $-1/0$ in place of $(n-1)/0$.

The group $\textnormal{SL}_2(\mathbb{Z}/n\mathbb{Z})$ acts on $\mathscr{E}_n$ by the rule
\[
\frac{x}{y}\longmapsto \frac{ax+by}{cx+dy},\quad\text{where }A=\begin{pmatrix}a &b \\ c&d\end{pmatrix}\in \textnormal{SL}_2(\mathbb{Z}/n\mathbb{Z}).
\]
This action is simply transitive on directed edges (see \cite[Proposition~2.3]{ShSoZa2024}) in the sense that, for any pair of directed edges $\gamma$ and $\delta$, there is a unique matrix $A$ that sends $\gamma$ to $\delta$.

Next we describe how to obtain $\mathscr{E}_{mn}$ from  $\mathscr{E}_{m}$ and $\mathscr{E}_{n}$ when $m$ and $n$ are coprime. For this we need to introduce tensor products of directed graphs. Consider two directed graphs $G$ and $H$, with vertex and directed edge sets $(V_G,E_G)$ and $(V_H,E_H)$. The \emph{tensor product} $G\otimes H$ of $G$ and $H$ is the directed graph with vertices $(u,v)$, where $u\in V_G$ and $v\in V_H$, and with a directed edge from $(u_1,v_1)$ to $(u_2,v_2)$ if and only if $u_1\to u_2$ belongs to $E_G$ and $v_1\to v_2$ belongs to $E_H$. (Here and below $x\to y$ denotes the directed edge from vertex $x$ to vertex $y$.)

In proving the next lemma (and later on) we write $a\bmod n$ for the integer in $\{0,1,\dots,n-1\}$ that is congruent to $a$ modulo $n$.

\begin{lemma}\label{lemma1}
Let $m$ and $n$ be coprime. Then $\mathscr{E}_{mn}\cong \mathscr{E}_{m}\otimes \mathscr{E}_{n}$.
\end{lemma}
\begin{proof}
Consider the map $\alpha\colon \mathscr{E}_{mn}\longrightarrow \mathscr{E}_{m}\otimes \mathscr{E}_{n}$ defined as follows. Given $a/b\in \mathscr{E}_{mn}$ we let $a_1= a \bmod m$, $b_1= b \bmod m$, $a_2= a \bmod n$, $b_2= b \bmod n$, and define $\alpha(a/b)=(a_1/b_1,a_2/b_2)$. It is straightforward to check that $\alpha$ is a well-defined graph homomorphism.

Next consider the map $\beta\colon \mathscr{E}_{m}\otimes \mathscr{E}_{n} \longrightarrow\mathscr{E}_{mn}$ defined as follows. For $a_1/b_1\in \mathscr{E}_m$ and $a_2/b_2\in \mathscr{E}_n$ we let $a\in\{0,1,\dots,mn-1\}$ be the unique solution of the congruences $x\equiv a_1 \pmod m$ and $x\equiv a_2 \pmod n$ and we let $b\in\{0,1,\dots,mn-1\}$ be the unique solution of the congruences $x\equiv b_1 \pmod m$ and $x\equiv b_2 \pmod n$, and then we define $\beta(a_1/b_1,a_2/b_2)=a/b$. Again, it is straightforward to check that $\beta$ is a well-defined graph homomorphism.  

A short calculation shows that $\alpha$ and $\beta$ are mutually inverse; hence $\mathscr{E}_{mn}\cong \mathscr{E}_{m}\otimes \mathscr{E}_{n}$.
\end{proof}

Our strategy involves lifting paths from $\mathscr{E}_{p}$ to $\mathscr{E}_{p^r}$, for a prime $p$ and positive integer $r$, and then using known results on enumerating paths in $\mathscr{E}_{p}$. For more delicate arguments we need to lift from $\mathscr{E}_{p^{s-1}}$ to $\mathscr{E}_{p^s}$ one stage at a time, for $s=2,3,\dots,r$.

Consider the graph homomorphism $\theta\colon \mathscr{E}_{p^r}\longrightarrow\mathscr{E}_{p^{r-1}}$ given by $\theta(a/b)= (a \bmod p^{r-1})/(b \bmod p^{r-1})$, which maps vertices $p^2$-to-1. It satisfies the equivariance property $\theta\circ A=\widehat{A}\circ \theta$, where $A\in \textnormal{SL}_2(\mathbb{Z}/p^r\mathbb{Z})$ and $\widehat{A}$ is the image of $A$  under the  homormophism $\textnormal{SL}_2(\mathbb{Z}/p^r\mathbb{Z})\longrightarrow \textnormal{SL}_2(\mathbb{Z}/p^{r-1}\mathbb{Z})$ given by reduction modulo $p^{r-1}$. A \emph{lift} to $\mathscr{E}_{p^r}$ of a vertex $v$ in $\mathscr{E}_{p^{r-1}}$ is a vertex $\bar{v}\in\mathscr{E}_{p^r}$ with $\theta(\bar{v})=v$. We use similar terminology for lifting directed edges and paths from $\mathscr{E}_{p^{r-1}}$ to $\mathscr{E}_{p^{r}}$ and from $\mathscr{E}_{p}$ to $\mathscr{E}_{p^{r}}$.

Fundamental to our strategy is the following basic path-lifting lemma.
\begin{lemma}\label{lemma2}
Let $\gamma$ be a path of length $m$ in $\mathscr{E}_{p^{r-1}}$ with initial vertex $v$, and let $\bar{v}$ be a lift of $v$ to $\mathscr{E}_{p^r}$. Then there are precisely $p^m$ different lifts of $\gamma$ to $\mathscr{E}_{p^r}$ with initial vertex $\bar{v}$.
\end{lemma}
\begin{proof}
Suppose first that $m=1$, in which case we are merely lifting a directed edge $v\to w$. After applying a suitable element of $\textnormal{SL}_2(\mathbb{Z}/p^r\mathbb{Z})$ we can assume that $v\to w$ is the directed edge $1/0\to 0/1$ (in $\mathscr{E}_{p^{r-1}}$) and $\bar{v}=1/0$ (in $\mathscr{E}_{p^r}$). Then there are precisely $p$ lifts of $\gamma$, namely the directed edges $1/0\to ap^{r-1}/1$, for $a=0,1,\dots,p-1$. For the general case, we simply apply this argument edge by edge to obtain $p^m$ lifts of $\gamma$.
\end{proof}

On one occasion later we apply Lemma~\ref{lemma2} in reverse form, where the final rather than the initial vertex of every lift of $\gamma$ is fixed. We also apply Lemma~\ref{lemma2} with the same hypotheses except that $\gamma$ lies in $\mathscr{E}_{p^s}$ rather than $\mathscr{E}_{p^{r-1}}$; in this case the number of lifts is $p^{(r-s)m}$, as we can see by lifting one stage at a time from $\mathscr{E}_{p^s}$ up to $\mathscr{E}_{p^r}$

\section{Proof of Theorem~A}

To prove Theorem~A we use Theorems~1.4 and~1.7 from \cite{ShSoZa2024}. The first of these two results is paraphrased in the following theorem, in which we say  that two vertices $u$ and $v$ of $\mathscr{E}_n$ are \emph{equivalent}  if $u=\lambda v$, for $\lambda\in(\mathbb{Z}/n\mathbb{Z})^\times$, the group of units of $\mathbb{Z}/n\mathbb{Z}$.

\begin{theorem}\label{theorem1}
There is a one-to-one correspondence between
\[
\hspace*{-10pt}\textnormal{SL}_2(\mathbb{Z}/n\mathbb{Z})\Big\backslash\mleft\{\parbox{4.25cm}{\centering\textnormal{paths~of length~$m$~between equivalent~vertices~in~$\mathscr{E}_{n}$}}\mright\}\quad \longleftrightarrow\quad (\mathbb{Z}/n\mathbb{Z})^\times\Big\backslash\mleft\{\parbox{2.8cm}{\centering\textnormal{tame~friezes~over $\mathbb{Z}/n\mathbb{Z}$~of~width~$m$}}\mright\}.
\]
\end{theorem}

This theorem is a special case of \cite[Theorem~1.4]{ShSoZa2024} for the ring $\mathbb{Z}/n\mathbb{Z}$ and Farey graph $\mathscr{E}_n$. Also, for convenience, we have framed this result in terms of tame friezes rather than tame \emph{semiregular} friezes by taking a quotient of $(\mathbb{Z}/n\mathbb{Z})^\times$ (see \cite{ShSoZa2024} for more on semiregular friezes).

The next theorem is the special case of \cite[Theorem~1.7]{ShSoZa2024} for the field $\mathbb{Z}/p\mathbb{Z}$.

\begin{theorem}\label{theorem2}
The number of tame  friezes of width $m$ over $\mathbb{Z}/p\mathbb{Z}$ is
\[
\frac{(p^{m-1}+(-1)^m)(p-1)}{p+1}.
\]
\end{theorem}

Let $\langle v_0,v_1,\dots,v_m\rangle$ denote the path in $\mathscr{E}_n$ with vertices $v_0,v_1,\dots,v_m$, in that order. We define $X_m(n)$ to be the collection of all paths of length $m$ in $\mathscr{E}_n$ with $v_0=1/0$, $v_1=0/1$, and $v_m$ equivalent to $v_0$. Since $\textnormal{SL}_2(\mathbb{Z}/n\mathbb{Z})$ acts simply transitively on directed edges in $\mathscr{E}_n$, the cardinality $|X_m(n)|$ of $X_m(n)$ is equal to that of 
\[
\textnormal{SL}_2(\mathbb{Z}/n\mathbb{Z})\Big\backslash\mleft\{\parbox{4.25cm}{\centering\textnormal{paths~of length~$m$~between equivalent~vertices~in~$\mathscr{E}_{n}$}}\mright\}.
\]
Theorem~\ref{theorem1} then tells us that the number of tame friezes of width $m$ over $\mathbb{Z}/n\mathbb{Z}$ is $\varphi(n)|X_m(n)|$, where $\varphi$ is Euler's totient function (and $\varphi(n)$ is the order of $(\mathbb{Z}/n\mathbb{Z})^\times$). By Lemma~\ref{lemma1}, the function $\varphi(n)|X_m(n)|$ is multiplicative in $n$, so it suffices to prove Theorem~\ref{theoremA} when $n$ is a prime power $p^r$.

\begin{lemma}\label{lemma3}
Any path in $X_m(p)$ has precisely $p^{(r-1)(m-2)}$ lifts to $X_m(p^r)$.
\end{lemma}
\begin{proof}
Let $\langle v_0,v_1,\dots,v_m\rangle$ be a path in $X_m(p)$, where $v_m=\lambda/0$ and $\lambda\neq 0$. By Lemma~\ref{lemma2}, there are $p^{(r-1)(m-2)}$ lifts $\langle 1/0,0/1,\bar{v}_2,\bar{v}_3,\dots,\bar{v}_{m-1} \rangle$  of $\langle v_0,v_1,\dots,v_{m-1} \rangle$ to $\mathscr{E}_{p^r}$. Since $v_{m-1}\to \lambda/0$ is a directed edge in $\mathscr{E}_p$, we see that $\bar{v}_{m-1}$ has the form $a/b$, where $b$ is a unit in $\mathbb{Z}/p^r\mathbb{Z}$. There is then precisely one directed edge  in $\mathscr{E}_{p^r}$ from $\bar{v}_{m-1}$ to a vertex  equivalent to $1/0$, namely $\bar{v}_{m-1}\to -b^{-1}/0$. Hence there are precisely $p^{(r-1)(m-2)}$ lifts, as required.
\end{proof}

Let us now complete the proof of Theorem~A. In the special case when $n$ is a prime $p$ we can apply Theorem~\ref{theorem2} to see that 
\[
|X_m(p)| = \frac{1}{\varphi(p)} \times \frac{(p^{m-1}+(-1)^m)(p-1)}{p+1}=\frac{p^{m-1}+(-1)^m}{p+1}.
\]
When $n$ is a prime power $p^r$ we can apply Lemma~\ref{lemma3} to give
\[
\varphi(p^r)|X_m(p^r)| =p^{r-1}(p-1)\times p^{(r-1)(m-2)}\times |X_m(p)| = \frac{p^{(r-1)(m-1)}(p^{m-1}+(-1)^m)(p-1)}{p+1}. 
\]
This completes the proof of Theorem~\ref{theoremA}.

\section{Proof of Theorem~B}

To prove Theorem~B we use Theorem~1.5 from \cite{ShSoZa2024}, stated below. This theorem uses the notion of a \emph{semiclosed} path in $\mathscr{E}_n$, which is a path with initial vertex $v$ and final vertex $-v$, for any vertex $v$ in $\mathscr{E}_n$.

\begin{theorem}\label{theorem3}
There is a one-to-one correspondence between
\[
\textnormal{SL}_2(\mathbb{Z}/n\mathbb{Z})\Big\backslash\mleft\{\parbox{3.1cm}{\centering\textnormal{semiclosed~paths~of length~$m$~in~$\mathscr{E}_n$}}\mright\}\quad \longleftrightarrow\quad \mleft\{\parbox{3.5cm}{\centering\textnormal{tame~regular~friezes over~$\mathbb{Z}/n\mathbb{Z}$~of~width~$m$}}\mright\}.
\]
\end{theorem}

Let $Y_m(n)$ denote the collection of paths in $\mathscr{E}_n$ with initial vertex $1/0$ and final vertex $-1/0$. Then, by Theorem~\ref{theorem3}, the number of tame regular friezes over~$\mathbb{Z}/n\mathbb{Z}$ of width~$m$ is $|Y_m(n)|/n$. Here the factor $n$ arises because we have freedom in choosing the second vertex under $\textnormal{SL}_2(\mathbb{Z}/n\mathbb{Z})$ equivalence (we elect not to specify that the second vertex is $0/1$ as we did for $X_m(n)$). By applying Lemma~\ref{lemma1} we can see that $|Y_m(n)|$ is a multiplicative function of $n$. Consequently, to prove Theorem~\ref{theoremB}, it sufficies to show that $|Y_m(p^r)|/p^r=\Phi_m(p^r)$ (using the notation of that theorem), for each prime power $p^r$. The remainder of this paper is dedicated to that task.

\begin{lemma}\label{lemma4}
Given any pair of vertices $a/b$ and $c/d$ in $\mathscr{E}_{p^r}$, where $b,c\not\equiv 0 \pmod{p}$ and at least one of $a,d\equiv 0\pmod{p}$, there is a unique vertex $v$ in $\mathscr{E}_{p^r}$ for which \(a/b \to v \to c/d\) is a path.
\end{lemma}
\begin{proof}
There is a path $a/b\to x/y\to c/d$ in $\mathscr{E}_{p^r}$ if and only if 
\[
ay-bx \equiv 1 \pmod{p^r}\quad\text{and}\quad
dx-cy \equiv 1 \pmod{p^r}.
\]
Since $b,c\not\equiv 0 \pmod{p}$ and one of $a,d\equiv 0\pmod{p}$ it follows that $ad-bc$ has a multiplicative inverse $\mu$ modulo $p^r$. With this observation, we can see that there is a unique solution to the pair of congruences, namely $x\equiv \mu(a+c) \pmod{p^r}$ and  $y\equiv \mu(b+d) \pmod{p^r}$, as required.
\end{proof}

A \emph{subpath} of a path $\langle v_0,v_1,\dots,v_m\rangle$ is a path  $\langle v_i,v_{i+1},\dots,v_j\rangle$, where $0\leq i<j\leq m$. We write $\ast$ for some unspecified vertex of whatever graph we are working with.

\begin{lemma}\label{lemma5}
Let $\gamma$ be a path in $Y_m(p^{s})$ that has a subpath of the form \(a/b \to \ast \to c/d\), where $b,c\not\equiv 0 \pmod{p}$ and one of $a,d\equiv 0\pmod{p}$. Then, for $r>s$, there are precisely $p^{(r-s)(m-2)}$ lifts of $\gamma$ to $Y_m(p^r)$.
\end{lemma}
\begin{proof}
Let $\gamma=\langle v_0,v_1,\dots,v_m\rangle\in Y_m(p^s)$. We can find an index $j$ with $v_{j-1}=a/b$ and $v_{j+1}=c/d$. By applying Lemma~\ref{lemma2}, in its normal form and in reverse form, we can find exactly $p^{(r-s)(m-2)}$ choices of vertices $\bar{v}_1,\bar{v}_2,\dots,\bar{v}_{j-1},\bar{v}_{j+1},\dots,\bar{v}_{m-1}$ in $\mathscr{E}_{p^r}$ such that $\langle \bar{v}_0,\bar{v}_1,\dots,\bar{v}_{j-1}\rangle$ is a lift of $\langle v_0,v_1,\dots,v_{j-1}\rangle$ and $\langle \bar{v}_{j+1},\bar{v}_{j+2},\dots,\bar{v}_{m}\rangle$ is a lift of $\langle v_{j+1},v_{j+2},\dots,v_{m}\rangle$ (where $\bar{v}_0=1/0$ and $\bar{v}_m=-1/0$). For any one of these choices of $m-2$ vertices, there is, by Lemma~\ref{lemma4}, a unique vertex $\bar{v}_j$ such that $\langle \bar{v}_0,\bar{v}_1,\dots,\bar{v}_{m}\rangle$ is a path -- and this path must be a lift of $\gamma$. Hence there are $p^{(r-s)(m-2)}$ lifts of $\gamma$, as required.
\end{proof}

The next lemma gives values of $m$ and $p$ for which \emph{all} paths in $Y_m(p)$ are of the type considered in Lemma~\ref{lemma5}.

\begin{lemma}\label{lemma6}
Suppose that either $m$ is odd or $m\equiv 0\pmod{4}$ and $p\neq 2$. Then any path $\gamma\in Y_m(p)$ has a subpath of one of the forms
\[
\frac{a}{b} \to \frac{c}{1}\to \frac{-1}{0}\quad\text{or}\quad\frac{a}{b} \to \frac{c}{-1}\to \frac{1}{0},
\]
where $a,b,c\in \mathbb{Z}/p\mathbb{Z}$ and $b\neq 0$. 
\end{lemma}
\begin{proof}
Let $\gamma=\langle v_0,v_1,\dots,v_m\rangle\in Y_m(p)$, and let us write $v_{m-2}=a/b$. If $b\neq 0$, then the final three vertices of $\gamma$ give a subpath of the required type. Suppose instead that $b=0$; then $a=1$. In this case the subpath $v_{m-4}\to v_{m-3}\to v_{m-2}$ has the form
\[
\frac{a'}{b'} \to \frac{c'}{-1}\to \frac{1}{0}.
\]
If $b'\neq 0$, then we have a subpath of the required type. Suppose instead that $b'=0$; then $a'=-1$. We can now repeat this argument, working backwards four edges at a time. If $m$ is odd, then this process must yield a subpath of the required type because $v_1=\lambda/1$, for $\lambda\neq 0$. The other possibility is that $m\equiv 0\pmod{4}$ and $p\neq 2$, and in this case the process must also yield a subpath of the required type because $v_0=1/0\neq -1/0$.
\end{proof}

Let $\Omega_m(p)$ be the collection of paths of even length $m=2k$ in $\mathscr{E}_p$ of the form
\[
\frac{1}{0} \to \frac{\lambda_1}{1} \to \frac{-1}{0} \to \frac{\lambda_2}{-1} \to\dotsb \to \frac{\varepsilon}{0}.
\]
The final vertex is $\varepsilon/0$, where $\varepsilon$ is $1$ if $k$ is even and $-1$ if $k$ is odd. For $m\equiv 2\pmod{4}$ (or $m\equiv 0\pmod{4}$ and $p=2$), the collection $\Omega_m(p)$ comprises those paths in $Y_m(p)$ \emph{not} of the type considered in Lemma~\ref{lemma6}. Counting the lifts of these paths to $Y_m(p^r)$ is the more challenging task that we now tackle.

For $1\leq t<r$, let $Z_k(r,t)$ denote the set of those lifts to $\mathscr{E}_{p^r}$ of paths from $\Omega_{m}(p)$ with initial vertex $1/0$ and final vertex of the form $(\varepsilon+a)/b$, where $a,b\equiv0\pmod{p}$, $\nu_p(b)=t$, and $\nu_p(a)\geq \nu_p(b)$ (and $\nu_p(0)$ is $\infty$). Let $Z_k(r)$ denote the set of those lifts to $\mathscr{E}_{p^r}$ of paths from $\Omega_{m}(p)$ with initial vertex $1/0$ and final vertex $\varepsilon/0$. We aim to count $Z_k(r)$.

\begin{lemma}\label{lemma7}
Suppose that  $a,b\equiv0\pmod{p}$ and $b\not\equiv 0\pmod{p^r}$. Let $s=\nu_p(a)$ and $t=\nu_p(b)$. Then the number of paths in $\mathscr{E}_{p^r}$ of the form
\[
\frac{-\varepsilon+a}{b} \to \ast \to \frac{\varepsilon}{0}
\]
 is zero if $s<t$ and $p^t$ if $s\geq t$.
\end{lemma}
\begin{proof}
There is a path of the given type if and only if the middle vertex has the form $x/(-\varepsilon)$ and \(bx \equiv -\varepsilon a \pmod{p^r}\). This final congruence has solutions if and only if $s\geq t$, and if $s\geq t$ then there are $p^t$ solutions given by $x \equiv -\varepsilon (a/p^t)(b/p^t)^{-1} \pmod{p^{r-t}}$.
\end{proof}

Consider the path $\gamma'$ obtained by removing the final two vertices from a path $\gamma\in Z_k(r)$, where $k>1$. The final vertex of $\gamma'$ has the form $(-\varepsilon+a)/b$, where $a,b\equiv 0\pmod{p}$. An elementary calculation shows that if $b\equiv 0\pmod{p^r}$, then $a\equiv 0\pmod{p^r}$ and there are $p^r$ paths of the form $-\varepsilon/0 \to \ast \to \varepsilon/0$. In this case $\gamma'\in Z_{k-1}(r)$. Alternatively, if $b\not\equiv 0\pmod{p^r}$, then Lemma~\ref{lemma7} tells us that $\nu_p(a)\geq \nu_p(b)$. In this case $\gamma'\in Z_{k-1}(r,t)$, where $t=\nu_p(b)$. Applying Lemma~\ref{lemma7} again we see that  
\begin{equation}\label{eqn2}
|Z_k(r)| = p^r|Z_{k-1}(r)|+\sum_{t=1}^{r-1} p^{t}|Z_{k-1}(r,t)|. 
\end{equation}

\begin{lemma}\label{lemma9}
For $k\geq 1$ and $r>1$ we have 
\begin{enumerate}
\item $|Z_{k}(r,t)| =p^{2k}|Z_{k}(r-1,t)|$, for $1\leq t<r-1$,
\item $|Z_{k}(r,r-1)| =p^{2k-1}(p-1)|Z_{k}(r-1)|$.
\end{enumerate} 
\end{lemma}
\begin{proof}
First we prove (i). Let $\gamma\in Z_k(r-1,t)$. Since $\gamma$ has length $2k$, we see from Lemma~\ref{lemma2} that there are precisely $p^{2k}$ lifts of $\gamma$ to $\mathscr{E}_{p^r}$ with initial vertex $1/0$. The condition $1\leq t<r-1$ ensures that each lift belongs to $Z_k(r,t)$. Hence $|Z_{k}(r,t)| =p^{2k}|Z_{k}(r-1,t)|$.

Next we prove (ii). Let $\gamma\in Z_k(r-1)$. There are $p^{2k}$ lifts of $\gamma'$ to $\mathscr{E}_{p^r}$ with initial vertex $1/0$. The final vertex of any lift has the form $(\varepsilon+ap^{r-1})/(bp^{r-1})$, where $a,b\in\{0,1,\dots,p-1\}$. One can check from the final edge that $a$ is uniquely specified by $b$. Now, this lift lies in $Z_k(r,r-1)$ if and only if $b\neq 0$ -- so there are $p^{2k-1}$ lifts of the first $2k$ vertices of $\gamma$ and $p-1$ suitable lifts of the last vertex. Hence $|Z_{k}(r,r-1)| =p^{2k-1}(p-1)|Z_{k}(r-1)|$. 
\end{proof}

From Lemma~\ref{lemma9} we have, for $k\geq 1$ and $1\leq t<r$,
\[
|Z_{k}(r,t)| =p^{2k(r-t-1)}|Z_{k}(t+1,t)| = p^{2k(r-t)-1}(p-1)|Z_{k}(t)|.
\]
Substituting this into \eqref{eqn2} gives
\[
|Z_k(r)| = p^r|Z_{k-1}(r)|+(p-1)p^{2r(k-1)-1}\sum_{t=1}^{r-1} p^{(3-2k)t}|Z_{k-1}(t)|. 
\]
One can then prove by induction (a task expedited with computer algebra software) that
\begin{equation}\label{eqn3}
|Z_k(r)|=p^{(r-1)(2k-2)+1}\mleft((p-1)[r-1]_{p^{2-k}}+p^{k-1}\mright),
\end{equation}
where the initial case $|Z_1(r)|=p^r$ is easily verified.



The set $Z_k(r)$ comprises lifts to $\mathscr{E}_{p^r}$ of paths from $\Omega_{2k}(p)$ with initial vertex $1/0$ and final vertex $1/0$ ($k$ even) or $-1/0$ ($k$ odd). It remains to count the set $W_k(r)$ of lifts to $\mathscr{E}_{p^r}$ of paths from $\Omega_{2k}(p)$ that have initial vertex $1/0$ and final vertex $-1/0$ when $k$ is even. This set is empty unless $p=2$.

\begin{lemma}\label{lemma10}
For $k$ even and $r>1$, 
\(
|W_k(r)|=2^{(r-2)(2k-2)}2^{2k-1}(2^{k-1}-1).
\)
\end{lemma}
\begin{proof}
Suppose that $r=2$. The vertex $1/0$ from $\mathscr{E}_2$ lifts to the set $V=\{1/0,-1/0,1/2,-1/2\}$ in $\mathscr{E}_4$, so all the even-index vertices of a path from $W_k(2)$ belong to $V$. To count $W_k(2)$, it is equivalent to count the number of paths of length $k$ from $1/0$ to $-1/0$ in the weighted graph $G$ with vertices $V$ and with weight for the edge between vertices $u$ and $v$ given by the number of paths of length 2 in $\mathscr{E}_4$ of the form $u\to \ast \to v$ (which is the same as the number of paths $v\to \ast \to u$). This graph is illustrated in Figure~\ref{fig2} alongside the adjacency matrix of the graph. Horizontal edges of the graph have weight 4 and vertical and diagonal edges have weight 2.
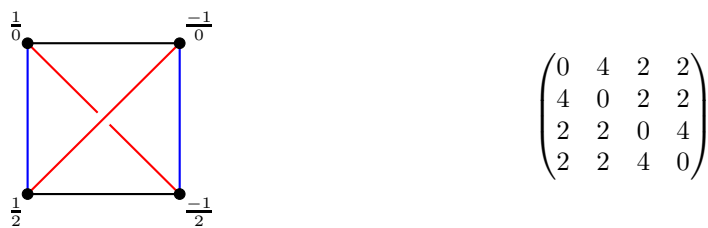
\begin{figure}[ht]
\centering
\begin{subfigure}[b]{0.5\textwidth}	
 \centering
\begin{tikzpicture}[scale=1,font=\small,inner sep=0.4mm]
	\coordinate (A) at (-1,1);
	\coordinate (B) at (1,1);
	\coordinate (C) at (1,-1);
	\coordinate (D) at (-1,-1);
		
	\draw[thick] (A)--(B);
	\draw[thick,red] (A)--(C);
	\filldraw[white] (0,0) circle (3pt);
	\draw[thick,blue] (A)--(D);
	\draw[thick,blue] (B)--(C);
	\draw[thick,red] (B)--(D);
	\draw[thick] (C)--(D);

	\filldraw[black] (A) circle (2pt) node[anchor=south east]{$\tfrac10$};
	\filldraw[black] (B) circle (2pt) node[anchor=south west]{$\tfrac{-1}0$};
	\filldraw[black] (C) circle (2pt) node[anchor=north west]{$\tfrac{-1}2$};
	\filldraw[black] (D) circle (2pt) node[anchor=north east]{$\tfrac12$};		
\end{tikzpicture}
\end{subfigure}
\begin{subfigure}[t]{0.4\textwidth}
 \centering
\raisebox{14mm}{\(
\begin{pmatrix}
0 & 4 & 2 & 2\\
4 & 0 & 2 & 2\\
2 & 2 & 0 & 4\\
2 & 2 & 4 & 0
\end{pmatrix}
\)}
\end{subfigure}
\caption{Graph $G$ (left) and its adjacency matrix (right)}
\label{fig2}
\end{figure}        

By taking the $k$th power of the adjacency matrix we can see that $|W_k(2)|=2^{2k-1}(2^{k-1}-1)$. We omit the details; the calculation can be verified with computer algebra software.

Now, observe that, because $k$ is even, any path $\gamma$ from $W_k(2)$, when considered as a path in $G$, must pass through a diagonal edge and a vertical edge, in some order, possibly with a number of horizontal edges in between. Let us assume that the diagonal edge comes first (the other case is similar). A quick check shows that diagonal edges correspond to paths $\ast \to \lambda/\mu \to \ast$ in $\mathscr{E}_4$ with $\lambda$ even and vertical edges correspond to paths of that form with  $\lambda$ odd. Consequently, there is a subpath of $\gamma$ of the form $a/b \to \ast \to c/d$, where $a$ is even (so $b$ is odd) and $c$ is odd. By Lemma~\ref{lemma5}, there are $2^{(r-2)(2k-2)}$ lifts of $\gamma$ to $W_k(r)$; hence $|W_k(r)|=2^{(r-2)(2k-2)}2^{2k-1}(2^{k-1}-1)$, as required.
\end{proof}

The final ingredient we need to prove Theorem~\ref{theoremB} is the following result of Morier-Genoud \cite{Mo2021} (see also \cite{ShSoZa2024}).

\begin{theorem}\label{theorem4}
The number of tame regular friezes of width $m$ over $\mathbb{Z}/p\mathbb{Z}$ is
\[
\Phi_m(p) = 
\begin{cases}
[k]_{p^2}, &\text{for $m=2k+1$,}\\
(p-1)\displaystyle\binom{k}{2}_{\!\!p} &\text{for $m=2k$ with $k$ even and $p\neq2$,}\\[1pt]
(p-1)\displaystyle\binom{k}{2}_{\!\!p}+p^{k-1} &\text{for $m=2k$ with $k$ odd or $p=2$}.
\end{cases}
\]
\end{theorem}

Let us complete the proof of Theorem~\ref{theoremB}. Theorem~\ref{theorem4} confirms the case $r=1$ from Theorem~\ref{theoremB}, so we assume instead that $r>1$. We must show that $|Y_m(p^r)|/p^r=\Phi_m(p^r)$ (which is true for $r=1$ by Theorem~\ref{theorem3}).

Suppose first that either $m$ is odd or $m\equiv 0\pmod{4}$ and $p\neq 2$ (the first two cases of Theorem~\ref{theoremB}). Then, by Lemmas~\ref{lemma5} and~\ref{lemma6}, we have $|Y_m(p^r)|=p^{(r-1)(m-2)}|Y_m(p)|$. Hence 
\[
\frac{|Y_m(p^r)|}{p^r}=\frac{p^{(r-1)(m-2)}|Y_m(p)|}{p^r}=p^{(r-1)(m-3)}|\Phi_m(p)|=|\Phi_m(p^r)|.
\]
Suppose instead that $m$ is even, and let $m=2k$. Assume for now that $k$ is odd (fourth case). We have $Y_{2k}(p^r)=Z_k(r)\cup Y_{2k}(p^r)\setminus Z_k(r)$, where $|Z_k(r)|$ is specified in \eqref{eqn3} and Lemma~\ref{lemma5} tells us that $|Y_{2k}(p^r)\setminus Z_k(r)|=p^{(r-1)(m-2)}|Y_{2k}(p)\setminus \Omega_{2k}(p)|$. Now, $|Y_{2k}(p)|=p|\Phi_{2k}(p)|$, so
\[
|Y_{2k}(p)\setminus \Omega_{2k}(p)|=|Y_{2k}(p)|-|\Omega_{2k}(p)|=\mleft(p(p-1)\displaystyle\binom{k}{2}_{\!\!p}+p^{k}\mright)-p^k=p(p-1)\displaystyle\binom{k}{2}_{\!\!p}.
\]
It follows that $|Y_{2k}(p^r)|/p^r=\Phi_{2k}(p^r)$.

Assume now that $k$ is even and $p=2$ (third case). We have $Y_{2k}(2^r)=W_k(r)\cup Y_{2k}(2^r)\setminus W_k(r)$, where $|W_k(r)|$ is specified in Lemma~\ref{lemma10} and, reasoning similarly to before, 
\[
|Y_{2k}(2^r)\setminus W_k(r)|=2^{(r-1)(m-2)}|Y_{2k}(2)\setminus \Omega_{2k}(2)|=2^{(r-1)(m-2)+1}\binom{k}{2}_{\!p}.
\]
 Once again we obtain $|Y_{2k}(p^r)|/p^r=\Phi_{2k}(p^r)$ (for $p=2$). This completes the proof of Theorem~\ref{theoremB}.

\end{document}